\documentclass[a4paper,reqno,11pt,oneside]{amsart}
\usepackage[all]{xy}
\usepackage{mmy}

\title[Derived equivalent non-birational Hilbert schemes]{Derived equivalent Hilbert schemes of points on K3 surfaces which are not birational}

\author{Ciaran Meachan}
\address{School of Mathematics, University of Glasgow, Scotland}
\email{ciaran.meachan@glasgow.ac.uk}
\author{Giovanni Mongardi}
\address{Dipartimento di Matematica, Universit\'a di Bologna, Italia}
\email{giovanni.mongardi2@unibo.it}
\author{K\={o}ta Yoshioka}
\address{Department of Mathematics, Kobe University, Japan}
\email{yoshioka@math.kobe-u.ac.jp}

\begin{document}

\maketitle

\begin{abstract}
We provide a criterion for when Hilbert schemes of points on K3 surfaces are birational. In particular, this allows us to generate a plethora of examples of non-birational Hilbert schemes which are derived equivalent.
\end{abstract}

\section*{Introduction}
The Bondal--Orlov conjecture \cite{bondal2002derived} provides a fundamental bridge between birational geometry and derived categories. It claims that if two varieties with trivial canonical bundle are birational then their bounded derived categories of coherent sheaves are equivalent. Whilst this conjecture is of paramount importance to the algebro-geometric community, it is examples where the converse fails that we are most interested in. The most famous example of this kind is Mukai's derived equivalence \cite{mukai1981duality} between an Abelian variety and its dual. Calabi--Yau examples have been the focus of a recent flurry of articles: \cite{borisov2009pfaffian,kuznetsov2007homological,kapustka2013mirror,addington2015pfaffian,ottem2017counterexample,borisov2017grassmannian,manivel2017double,kapustka2017torelli}, but there were no such examples in the hyperk\"ahler setting until very recently. 
Indeed, the first examples of derived equivalent non-birational hyperk\"ahlers were exhibited in \cite[Theorem B]{addington2016moduli} as certain moduli spaces of torsion sheaves on K3 surfaces. This article complements this discovery with further examples coming from Hilbert schemes of points and, in some sense, completes the investigation initiated by Ploog in \cite{ploog2007equivariant}. 

\subsection*{Various notions of equivalence}

Throughout this article, we will use $\D(X)$ to denote the bounded derived category of coherent sheaves on a smooth complex projective variety $X$. Moreover, we will say that two smooth complex projective varieties $X$ and $Y$ are D-equivalent if we have an equivalence $\D(X)\simeq\D(Y)$. 

Recall that two varieties $X$ and $Y$ are said to be K-equivalent if there exists a birational correspondence $X\xla{\pi_X} Z\xra{\pi_Y} Y$ with $\pi_X^*\omega_X\simeq\pi_Y^*\omega_Y$. The Bondal--Orlov conjecture says $\K\implies\D$. Notice that if the canonical bundles of $X$ and $Y$ are trivial then K-equivalence is the same as birationality.

We will say that two varieties $X$ and $Y$ are H-equivalent if there exists a Hodge isometry $\H^2(X,\bbZ)\simeq_\Hdg\H^2(Y,\bbZ)$, that is, an isomorphism respecting the Hodge structure and the intersection pairing. For a general hyperk\"ahler, Huybrechts \cite[Corollary 4.7]{huybrechts1999compact} shows that $\K\implies\H$. However, Namikawa \cite{namikawa2002counter} showed that there are Abelian surfaces $A$ for which the associated generalised Kummer fourfolds $K_2(A)$ and $K_2(\hat{A})$ are Hodge-equivalent but not birational; this was the first major counterexample to the birational Torelli problem for hyperk\"ahlers. Also, Verbitsky's Torelli theorem \cite[Theorem 7.19]{verbitsky2013torelli} proves that $\H\implies\K$ when the hyperk\"ahler is of $\trm{K3}^{[n]}$-type and $n=p^k+1$ for some prime $p$ and positive integer $k$. However, if $n-1$ is not a prime power then Markman \cite[Lemma 4.11]{markman2010integral} provides examples of Hilbert schemes which are Hodge-equivalent and yet not birational. Taken together, these results illustrate 
the relationship between H-equivalence and K-equivalence of hyperk\"ahlers is quite delicate. That is, while Huybrechts says that $\K\implies\H$, the converse only holds when we impose certain extra conditions.

If $X$ is a hyperk\"ahler of $\trm{K3}^{[n]}$-type then the Markman--Mukai lattice $\tilde\Lambda$ is an extension of the lattice $\H^2(X,\bbZ)$ and the weight-two Hodge structure on $\H^2(X,\bbC)$ with the following properties:
\begin{enumerate}
\item as a lattice, we have $\tilde\Lambda\simeq U^4\oplus E_8(-1)^2$,
\item the orthogonal complement of $\H^2(X,\bbZ)$ in $\tilde\Lambda$ has rank one and is generated by a primitive vector of square $2n-2$,
\item if $X$ is a moduli space $M_S(\vv)$ of sheaves on a K3 surface $S$ with Mukai vector $\vv\in\H^*(S,\bbZ)$ then there is an isomorphism $\tilde\Lambda\simeq\H^*(S,\bbZ)\;;\;\H^2(X,\bbZ)\mapsto \vv^\perp$.
\end{enumerate}

We say that two hyperk\"ahlers $X$ and $Y$ are M-equivalent if there exists a Hodge isometry $\tilde\Lambda_X\simeq_\Hdg\tilde\Lambda_Y$. Markman's Torelli theorem \cite[Corollary 9.9]{markman2011survey} shows that M-equivalence almost implies birationality. More precisely, if $X_1$ and $X_2$ are two hyperk\"ahlers of $\trm{K3}^{[n]}$-type with an M-equivalence $\phi:\tilde\Lambda_1\simeq_\Hdg\tilde\Lambda_2$ then $X_1$ and $X_2$ are birational if and only if $\phi$ maps $\H^2(X_1,\bbZ)$ to $\H^2(X_2,\bbZ)$. Moreover, if $X_1$ and $X_2$ are both moduli spaces of sheaves on K3 surfaces $S_1$ and $S_2$ with Mukai vectors $v_i\in\H^*(S_i,\bbZ)$, $i=1,2$, then we can use property (iii) above to rephrase Markman's Torelli theorem as follows: $X_1$ and $X_2$ are birational if and only if there exists an M-equivalence $\phi:\H^*(S_1,\bbZ)\simeq_\Hdg\H^*(S_2,\bbZ)\;;\;\vv_1\mapsto \vv_2$. In particular, in all our examples of non-birational derived equivalent Hilbert schemes, we have plenty of M-equivalences but none of them preserving $(1,0,1-n)$. It is tempting to speculate that D-equivalence for hyperk\"ahlers is 
implied by M-equivalence: 
\begin{equation}\label{eqn:relations}
\begin{aligned}
\xymatrix{\K \ar@{==>}[r]^-{\text{conj}} \ar@{=>}[d] \ar@{=>}[dr] & \D \\ \H \ar@{=>}[r] & \M \ar@{==>}[u]_-?}
\end{aligned}
\end{equation}
%
%

In the special case of Hilbert schemes of points on K3 surfaces, the implication $\K\implies \D$ was established in \cite[Proposition 10]{ploog2007equivariant}. Halpern-Leistner \cite{halpern-leistner2017KimpliesD} has announced a generalisation of this result to moduli spaces of sheaves on K3 surfaces.
\medskip

\textbf{Acknowledgements}: A very similar result was independently discovered by Shinnosuke Okawa \cite{okawa2018example}. His result only considers the case when Brill--Noether contractions exist. The first author thanks David Ploog for a very helpful and enjoyable discussion. Special thanks also go to Evgeny Shinder and Joe Karmazyn for pointing out the interesting Example \ref{5H13}.

\section{Examples which are D-equivalent but not K-equivalent}

\subsection{
Degree Twelve} \label{sec:deg12}
We work through a specific example in order to demonstrate how certain Hilbert schemes can be derived equivalent and not birational.
\vs

Let $X$ be a complex projective K3 surface with $\Pic(X)=\bbZ[H]$ and $\ww\in\H^*_\alg(X,\bbZ)$ a primitive vector with $\ww^2=0$. Then Mukai \cite{mukai1987moduli} shows that the moduli space $Y=M_H(\ww)$ of Gieseker $H$-stable sheaves is a K3 surface. Moreover, the derived Torelli theorem of Mukai and Orlov \cite{orlov1997equivalences} shows that if there exists a vector $\vv\in\H^*_\alg(X,\bbZ)$ with $(\vv,\ww)=1$\footnote{The condition that $(\vv,\ww)=1$ is equivalent to Mukai's criterion that $\ww=(r,H,s)$ for some integers $r$ and $s$ satisfying $\gcd(r,s)=1$ and $H^2=2rs$.} then there is a universal family $\cE$ on $X\times Y$ which induces a derived equivalence: \[\cF_\cE:\D(X)\xra\sim\D(Y).\] By \cite[Proposition 8]{ploog2007equivariant}, this gives equivalences $\D(X^{[n]})\simeq\D(Y^{[n]})$ for all $n\ge1$. 

\begin{ques}\label{biratQ} 
For which positive integers $n$, are $X^{[n]}$ and $Y^{[n]}$ birational?
\end{ques}

Recall that Oguiso \cite{oguiso2002K3} has shown that the number of Fourier--Mukai partners of a K3 surface $X$ with $\Pic(X)=\bbZ[H]$ and $H^2=2d$ is given by $2^{\rho(d)-1}$, where $\rho(d)$ is the number of prime factors of $d$. Thus, to ensure that $X^{[n]}$ and $Y^{[n]}$ are not all isomorphic, we must have $H^2\ge12$. For simplicity, we choose $H^2=12$. 

Let $X$ be a complex projective K3 surface with $\Pic(X)=\bbZ[H]$ and $H^2=12$. Since $\ww=(2,H,3)$ is an isotropic vector, we have another K3 surface $Y=M_H(\ww)$. Moreover, since $\vv=(1,H,4)$ is a vector such that $(\vv,\ww)=1$ (or $\gcd(2,3)=1$), we have a universal family $\cE$ and a derived equivalence as above. Now, the proof of \cite[Theorem 2.4]{stellari2004some} shows that $\H^2(X,\bbZ)\not\simeq_\Hdg\H^2(Y,\bbZ)$ and so the K3 surfaces $X$ and $Y$ are not birational. That is, when $n=1$ the answer to our Question \ref{biratQ} above is no. However, when $n=2,3,4$ the answer to Question \ref{biratQ} is yes! 

To see this, we use \cite[Lemma 7.2]{yoshioka2001moduli} (with $d_0,d_1,l=1$, $r_0=2$ and $k=3$) which shows that the cohomological Fourier--Mukai transform acts as follows:
\[\cF_\cE^\H:\H^*(X,\bbZ)\to\H^*(Y,\bbZ)\;;\;\left\{\begin{array}{ccccc}(1,0,0)&\mapsto&(3,\hat{H},2)\\(0,H,0)&\mapsto&(12,5\hat{H},12)\\(0,0,1)&\mapsto&(2,\hat{H},3)&=&\ww,\end{array}\right.\] 
where $\hat{H}$ is an ample divisor class on $Y$. In particular, since $\cF_\cE^\H$ is a Hodge isometry and $X^{[2]}\simeq M_X(1,0,-1)$, we have 
\[\cF_\cE^\H(1,0,-1)=\cF_\cE^\H(1,0,0)-\cF_\cE^\H(0,0,1)=(1,0,-1),\]
and hence a birational map $\cF_\cE:
X^{[2]}\dra Y^{[2]}$. In this case, \cite[Theorem 7.6]{yoshioka2001moduli} shows that $\cF_\cE$ is actually an isomorphism! See \cite[Example 7.2]{yoshioka2001moduli} for details.

Similarly, when $n=3$ we have $X^{[3]}\simeq M_X(1,0,-2)\simeq M_X(1,-H,4)$, where the second isomorphism is given twisting by $\cO_X(H)$. Thus, we see that 
\[\cF_\cE^\H(1,-H,4)=\cF_\cE^\H(1,0,0)-\cF_\cE^\H(0,H,0)+4\cF_\cE^\H(0,0,1)=-(1,0,-2),\]
and hence $\cF_\cE[1]:M_X(1,-H,4)\dra M_Y(1,0,-2)$ is a birational map $X^{[3]}\dra Y^{[3]}$. 

For $n=4$, 
$X^{[4]}\simeq M_X(1,-H,3)$, and it is enough, by \cite[Corollary 1.3]{bayer2013mmp}, to find an equivalence $\Phi:\D(X)\xra\sim\D(Y)$ such that $\Phi^\H$ maps $(1,-H,3)$ to $(1,\hat{H},3)$. If we set $\bbD_X:=\cHom(\_,\omega_X)[2]$ to be the dualising functor and $T_{\cO_X}$ to be the spherical twist around $\cO_X$ then $\Phi:=T_{\cO_Y}\circ\bbD_Y\circ\cF_\cE$ does the job. Indeed, since $T_{\cO_X}^\H$ sends a class $(r,c,s)$ to $(-s,c,-r)$, we can check that we have
\[(1,-H,3)\xra{\cF_\cE^\H}(-3,-\hat{H},-1)\xra{\bbD_Y^\H}(-3,\hat{H},-1)\xra{T_{\cO_X}^\H}(1,\hat{H},3).\]

For $n=5$, we 
first note that $X^{[5]}\simeq M_X(1,-H,2)$ is birational to $M_X(2,H,1)$ 
and then observe that the second moduli space has a Li--Gieseker--Uhlenbeck contraction. Indeed, the Hodge isometry $T_{\cO_X}^\H[1]$ sends $(1,-H,2)$ to $(2,H,1)$ and the Mukai vector $\ww=(0,0,-1)$ is an isotropic class which pairs with $(2,H,1)$ to give 2. 
Now, if $X^{[5]}$ and $Y^{[5]}$ are birational then we have an induced map between their second integral cohomology groups. Moreover, since a birational map preserves the movable cone (cf. \cite[Section 6]{markman2011survey}), this map can either send the exceptional divisor of the Hilbert--Chow (HC) contraction 
to itself or to the exceptional divisor of the Li--Gieseker--Uhlenbeck (LGU) contraction. 
In particular, if HC were mapped to LGU then we would contradict the fact that a birational map between Hilbert schemes necessarily sends primitive classes to primitive classes, whereas if HC were mapped to HC then the orthogonal complements must be Hodge-isometric as well, i.e. $\H^2(X,\bbZ)\simeq_\Hdg\H^2(Y,\bbZ)$, and hence the underlying K3s would be isomorphic which they are not. Thus, by contradiction, $X^{[5]}$ and $Y^{[5]}$ cannot be birational.

The key thing about the previous argument is that the movable cone of $X^{[5]}$ has two \emph{different} boundary walls. For a Picard rank one K3 surface $X$, the movable cone of the Hilbert scheme $X^{[n]}$ has two boundary walls. At least one of these boundaries is a Hilbert--Chow wall, and so we are essentially looking to see if the other wall is a different type: Brill--Noether (BN), Li--Gieseker--Uhlenbeck (LGU), or Lagrangian fibration (LF). If it is then a similar argument to case of $n=5$ above shows that $X^{[n]}$ and $Y^{[n]}$ cannot be birational, where $Y=M_X(2,H,3)$.

%
\begin{prop}\label{prop:deg12}
If $(X,H)$ is a polarised K3 surface with $H^2=12$ then the moduli space $Y=M_X(2,H,3)$ is another K3 surface with $\D(X^{[n]}) \simeq \D(Y^{[n]})$ for all $n\ge 1$. Moreover, 
the Hilbert schemes $X^{[n]}$ and $Y^{[n]}$ are birational if and only if there is a solution to either of the Pell's equations: \[2(n-1)x^2-3y^2=\pm1\qquad\trm{or}\qquad 3(n-1)x^2-2y^2=\pm1.\] In particular, when $n=5,6,7,8,9,11,
\dots$, these Hilbert schemes are not birational.
\end{prop}

\begin{proof}
As discussed above, the derived Torelli theorem of Mukai and Orlov \cite{orlov1997equivalences} shows that $\D(X)\simeq\D(Y)$ and hence Ploog's result \cite[Proposition 8]{ploog2007equivariant} ensures that $\D(X^{[n]})\simeq\D(Y^{[n]})$ for all $n\ge1$. The conditions on when the Hilbert schemes are birational can be found in Theorem \ref{prop:Torelli} below.
\end{proof}

\begin{rmk}\label{rmk:summarydeg12}
We can summarise 
part of Proposition \ref{prop:deg12}, and in particular the answers to Question \ref{biratQ} for the example when $X=\text{K3}$, $\Pic(X)=\bbZ[H]$, $H^2=12$, and $Y=M_X(2,H,3)$, with the following table:
%


\begin{center}
\begin{tabular}{ |c|c|c|c|c|c|c|c|c|c|c|c|c|c|c|c|c|c|c|c|c| } \hline
$n$ &1&2&3&4&5&6&7&8&9&10&11&12&13&14&15&16&17\\  \hline
birational? & \xmark& \cmark& \cmark& \cmark& \xmark& \xmark& \xmark& \xmark& \xmark& \cmark& \xmark& \cmark& \xmark& \cmark& \cmark& \xmark& \xmark\\  \hline
\end{tabular}
\end{center} 
\end{rmk}

%

\begin{rmk}
Notice that 
the Hilbert schemes $X^{[n]}$ and $Y^{[n]}$ in Proposition \ref{prop:deg12} are all M-equivalent. Indeed, the 
Markman--Mukai lattices $\tilde\Lambda$ of $X^{[n]}$ and $Y^{[n]}$ are given by $\H^*(X,\bbZ)$ and $\H^*(Y,\bbZ)$, respectively; these are Hodge isometric by \cite{orlov1997equivalences}. 
When $n=p^k+1$ for some positive integer $k$, 
we can use Verbitsky's Torelli theorem to conclude that the non-birational examples are also not H-equivalent. For example, when $n=5,6,8,9,17,...$ the examples in Remark \ref{rmk:summarydeg12} are not H-equivalent, but for $n=7,11,13,16,19,\dots$ Torelli does not hold and so we cannot say whether they are H-equivalent or not. This information should be taken into consideration along with the diagram in \eqref{eqn:relations}. Finally, let us point out that all of these pairs of Hilbert schemes are also deformation equivalent to each other, and thus complement the recent articles mentioned in the introduction.
\end{rmk}
\vs

\subsection{K3s with many FM partners}

A second way to produce examples is by considering K3 surfaces with many Fourier--Mukai partners. However, this produces less ``constructible'' examples, as the following shows.

\begin{prop}
Let $(X,H)$ be a K3 surface of degree $2pqr$ and Picard rank one, where $p,q,r$ are relatively prime integers greater than 1. Then, for all $n$, there exists a K3 surface $Y$ such that $\D(X^{[n]})\simeq \D(Y^{[n]})$ and $X^{[n]}$ is not birational to $Y^{[n]}$.
\end{prop}
\begin{proof}
From the condition on the degree of the polarisation of $X$, it follows from \cite{oguiso2002K3} that $X$ has 
four 
non-isomorphic Fourier--Mukai partners: $X,Y,Z,W$. Since $X$ has Picard rank one, we see that for every $n>1$, the movable cone of $X^{[n]}$ has exactly two extremal rays. Suppose, for a contradiction, that \emph{three} of the four Hilbert schemes $X^{[n]},Y^{[n]},Z^{[n]}$ and $W^{[n]}$ are all birational. 
Then these three birational maps induce maps between the movable cones of these Hilbert schemes, sending extremal rays to extremal rays. In particular, it follows from the arguments given in Section \ref{sec:deg12} that two of these Hilbert schemes must have Hodge isometric second integral cohomology groups with an isometry preserving the class of the exceptional divisor. Hence, the second integral cohomology groups of the underlying K3s must be Hodge isometric as well, meaning that two of the Fourier--Mukai partners are isomorphic, which they are not. Thus, by contradiction, two of the Hilbert schemes cannot be birational for any $n\ge1$.
\end{proof}
%

\section{
Criterion For Birationality of Hilbert Schemes}
We give a criterion for when Hilbert schemes of points on certain K3 surfaces are birational. More specifically, given a K3 surface $X$ and a FM-partner $Y=M_X(\vv)$, we provide a criterion which determines precisely when $X^{[n]}$ is birational to $Y^{[n]}$.
First of all, let us start by recalling some properties of moduli spaces on K3 surfaces of Picard rank one:
let $(X,H)$ be a polarised K3 surface such that $\Pic(X)=\bbZ H$.
 
\begin{prop}\label{prop:Partner}
Let $\vv=(r,cH,x)$ be a  primitive isotropic Mukai vector. Then the following holds:
\begin{enumerate}
\item There exist integers $p,s,q,t$ such that $(r,cH,x)=(p^2 s,pq H,q^2 t)$, where $H^2=2st$ and $\gcd(p,q)=1$.
\item If $M_X(r,cH,x)$ is fine then $(r,cH,x)=(p^2 s,pq H,q^2 t)$ with $\gcd(ps,qt)=1$. Moreover, in this case $M_X(p^2 s,pqH,q^2 t) \simeq M_X(s,H,t)$.
\item $M_X(s,H,t) \simeq M_X(s',H,t')$ if and only if $\{s,t\}=\{s',t'\}$.
\end{enumerate}
\end{prop}

\begin{proof}
We set $p:=\gcd(r,c)$. Since $v$ is primitive and isotropic, we have $\gcd(p,x)=1$ and $c^2 H^2/2=rx$, respectively. Thus, we see that $p^2 \mid r$. If we set $r=p^2 s$ and $c=pq$ then we must have $q^2H^2/2=sx$ and $\gcd(p^2 s,pq)=p$. This implies that we have $\gcd(q,s)=1$, and hence $x=q^2 t$ and $H^2/2=st$. 

Recall from \cite[Corollary 4.6.7]{huybrechts2010geometry} that $M_X(r,cH,x)$ is a fine moduli space if and only if $\gcd(r,cH^2,x)=1$. Hence $\gcd(r,x)=1=\gcd(ps,qt)$.  

By \cite[\S 5.2 and Prop. 5.6]{shouhei2008on}, the Mukai vector $(p^2 s,pq H,q^2 t)=p^2 s \exp(\tfrac{q}{ps}H)$ corresponds to $(k,s)$ in \cite[Equation (21)]{shouhei2008on}, where $k$ is any integer. Since this identification is independent of $p,q$, we get $M_X(p^2 s,pqH,q^2 t) \simeq M_X(s,H,t)$. The last claim is the content of \cite{hosono2003fourier} (see also \cite[Prop. 5.6]{shouhei2008on}).
\end{proof}

The previous proposition, together with Verbitsky's global Torelli theorem \cite{verbitsky2013torelli}, Markman's computation of the monodromy group \cite{markman2010integral}, and Bayer and Macr\`i's results about the ample cone of moduli spaces \cite{bayer2013mmp}, gives the following:

\begin{thm}\label{prop:Torelli}
Let $X$ and $Y$ be two derived equivalent K3 surfaces of Picard rank one. Then, $X^{[n]}$ is birationally equivalent to $Y^{[n]}$ if and only if $p^2 s(n-1)-q^2 t=\pm 1$ and $Y=M_X(p^2 s,pq H,q^2 t)$. Moreover, $\{s,t\}$ is uniquely determined by $Y$.
\end{thm}

\begin{proof}
By the global Torelli theorem for irreducible symplectic manifolds, $X^{[n]}$ and $Y^{[n]}$ are birational if and only if they are Hodge isometric through a monodromy operator. By Markman's computation of the monodromy groups (see \cite[Corollary 1.3]{bayer2013mmp}), this means that such an isometry extends to the Mukai lattice associated to the two K3s $X$ and $Y$. Thus, $X^{[n]}$ is birationally equivalent to $Y^{[n]}$ if and only if there is a primitive isotropic Mukai vector $\ww=\pm (p^2 s,pq H,q^2 t)\in \H^*(X,\bbZ)$ such that $\langle (1,0,1-n),\ww \rangle =1$ and $Y \simeq M_X(\ww)$. The first condition is equivalent to $p^2 s(n-1)-q^2 t=\pm 1$ and, by Proposition \ref{prop:Partner}, the pair $\{s,t\}$ is determined by $Y$. Notice that the Mukai vector $\ww$ will correspond to the Hilbert--Chow (birational) contraction on $X^{[n]}$ which has $Y^{(n)}$ as the base variety.
\end{proof}

\begin{rmk}
Notice that because the pair $\{s,t\}$ is determined by $M_X(p^2s,pqH,q^2t)$ and $M_X(s,H,t)\simeq M_X(t,H,s)$ by Proposition \ref{prop:Partner}(iii), we actually have two Pell's equations governing the birationality. That is, if we want to check whether two Hilbert schemes are birational then we need to find a solution to either: \[p^2 s(n-1)-q^2 t=\pm 1\qquad \text{or} \qquad p^2 t(n-1)-q^2 s=\pm 1.\]
\end{rmk}

Let us look back at the case analysed in Proposition \ref{prop:deg12}:
\begin{exa}
If $H^2=12$ then the only Fourier--Mukai partner of $X$, other than itself, is given by $Y:=M_X(2,H,3)$. 
For $n=5$, it is easy to see that there is no solution to $8p^2-3q^2=\pm 1$ or $12p^2-2q^2=\pm 1$. Hence, $X^{[5]}$ and $Y^{[5]}$ are not birationally equivalent. 
For $n=10$, we have $\langle (1,0,-9),(27,33H,242) \rangle=1$, and so $X^{[10]}$ and $Y^{[10]}$ are birational.
Similarly, for $n=12$, we can observe that $\langle(1,0,-11),(3,4H,32) \rangle=1$, hence $X^{[12]}$ and $Y^{[12]}$ are birational.


\end{exa}

\begin{exa}\label{5H13}
If $H^2=130=2\cdot5\cdot13$ then we have two Fourier--Mukai partners and the non-trivial partner is given by $Y=M_X(5,H,13)$. By Theorem \ref{prop:Torelli}, we see that $X^{[n+1]}$ is birational to $Y^{[n+1]}$ if and only if there is a solution to either: 
\[5np^2-13q^2=\pm 1\qquad \text{or} \qquad 13np^2-5q^2=\pm 1.\]
Reducing these equations modulo 5 and 13, respectively, we see that there are no solutions to $13q^2=\pm 1\mod5$ or $5q^2=\pm 1\mod 13$. This shows that there are no solutions to the original equations for any $n$. In other words, whilst these Hilbert schemes are always derived equivalent, they are, in fact, never birational!
\end{exa}

\subsection{Counting birational equivalence classes}\label{subsec:bircounting}
An interesting question concerns the number of non-birational derived equivalent Hilbert schemes that we can produce starting from the set of Fourier--Mukai partners of $X$. As we analysed in the previous sections, the two numbers are strictly linked: for any $X$ as above, the Hilbert scheme $X^{[n]}$ has precisely two boundaries of the movable cone and they only depend on the algebraic part of its Hodge structure, hence the Hilbert scheme $Y^{[n]}$ on any Fourier--Mukai partner $Y$ of $X$ has the same geometry of rays making up the movable cone, see \cite[Prop. 13.1]{bayer2013mmp}. If $X^{[n]}$ and $Y^{[n]}$ are birational for two different Mukai partners $X$ and $Y$, then one ray from each cone has to correspond to the Hilbert--Chow contractions. Therefore, if $N$ is the number of Fourier--Mukai partners of $X$, the number $B$ of birational equivalence classes of Hilbert schemes of points on these partners is either $N$ or $N/2$. Indeed, the former occurs when $X^{[n]}$ is not birational to any other $Y^{[n]}$, and the latter occurs when there is a single $Y^{[n]}$ birational to $X^{[n]}$ which represents the second Hilbert-Chow wall. Note that as soon as $X^{[n]}$ is birational to $Y^{[n]}$ for a single Fourier--Mukai partner $Y$ of $X$ then the same happens for all Fourier--Mukai partners of $X$.

When $N=B$, we have one of the following:
\begin{itemize}
\item There is a Hilbert--Chow wall and a different divisorial contraction on $X^{[n]}$.
\item $X^{[n]}$ has a Lagrangian fibration.
\item There are two Hilbert--Chow walls in the movable cone of $X^{[n]}$ which are exchanged by a birational map. 
\end{itemize}  
To state the result properly, we need to introduce a few more notations and results contained in \cite{yanagida2014bridgeland}. In \emph{loc.\;cit.}, the results are stated for Abelian surfaces but they still hold for K3 surfaces {\it mutatis mutandis}. 
We assume that $\sqrt{ (n-1)d} \not \in \bbZ $, where $d=H^2/2$ is half the degree of the K3 surface as before and $n>2$.

\begin{dfn}
For $(x,y)\in \bbR^2$, set
\begin{align*}
P(x,y):=\begin{pmatrix}y&(n-1) x\\ x&y\end{pmatrix}.
\end{align*}
We also set
\begin{align*}
S_{d,n}:=\left\{\begin{pmatrix}y&(n-1) x\\ x&y\end{pmatrix}\,\Bigg|\,
\begin{aligned}
x=a \sqrt{s}, y=b \sqrt{t},\;
a,b,s, t \in\bbZ\\
s,t>0,\; st=d,\; 
y^2-(n-1) x^2=\pm 1 
\end{aligned}
\right\}.
\end{align*}
\end{dfn}
The group we just defined has the following structure:
\begin{lem}
If $n>2$, then $S_{d,n}/\pm 1$ is an infinite cyclic group.
\end{lem}
\begin{proof}See \cite[Corollary 6.6]{yanagida2014bridgeland}.\end{proof}
The key use of this group is that its action allows us to determine different presentations of the Mukai vector $(1,0,1-n)$ corresponding to the Hilbert scheme of points on $X$ as a sum of two isotropic vectors (which will correspond to the Mukai vectors $(1,0,0)$ and $(0,0,1-n)$ on a Fourier--Mukai partner of $X$), as proven in \cite[Lemma 6.8]{yanagida2014bridgeland}.
The number $B$ then depends on a generator of $S_{d,n}/\pm 1$: 
\begin{prop}
We set $N:=2^{\rho(d)-1}$, where $\rho(d)$ is the number of prime divisors of $d$, and assume $n>2$. Let $X_1,...,X_N$ be all the Fourier--Mukai partners of $X$ and $$\cB:=\{{X_i}^{[n]} \mid i=1,...,N \}/\sim$$ be the set of birational equivalence classes of ${X_i}^{[n]}$. Then $|\cB|$ is either $N$ or $N/2$.
\end{prop}
\begin{proof}
Let $P(a \sqrt{s},b\sqrt{t}) \in S_{d,n}$ be a generator of $S_{d,n}/\pm 1$. Then, for any $P(x,y) \in S_{d,n}$,
$$
(x,y)=(a'\sqrt{s},b'\sqrt{t}) \text{ or } (a' \sqrt{d},b').
$$
Hence ${X_i}^{[n]}$ is birationally equivalent to ${X_j}^{[n]}$ if and only if $X_i=
M_{X_j}(a^2 s,abH,b^2 t)$. In particular, we have
\begin{equation}\label{eq:B}
|\cB|=
\begin{cases} N/2, & \{s,t \} \ne \{1,d\}\\
N, & \{s,t \} = \{1,d\}
\end{cases}
\end{equation}
using the same arguments as presented at the start of Section \ref{subsec:bircounting}.
\end{proof}

\begin{exa}
If $n=2$ then $S_{d,n} \simeq (\bbZ/2\bbZ)^{\oplus 2} \oplus \bbZ$ and the torsion subgroup is 
\[\left\{\pm \begin{pmatrix}1 & 0\\ 0 & 1 \end{pmatrix},\pm \begin{pmatrix}0 & 1\\ 1 & 0 \end{pmatrix}\right\}=\left\{\pm P(1,0),\pm P(0,1)\right\}.\] 
For a generator $P(a \sqrt{s},b\sqrt{t})$ of a cyclic subgroup,
we have a similar claim to \eqref{eq:B}.
\end{exa} 

\begin{exa}
If there are integers $p,q$ satisfying $dp^2-(n-1)q^2=\pm 1$, then $|B|=N$. In particular, if $n-1=dp^2 \pm 1$, then $|B|=N$.
\end{exa}

\begin{rmk}
Assume that $\sqrt{d(n-1)} \in \bbZ$. In this case, $p^2s(n-1)-q^2 t=\pm 1$ implies $\gcd(s (n-1),t)=1$. Hence $\sqrt{s (n-1)},\sqrt{t} \in \bbZ$. Then $p=0$ and $q^2=t=1$, or $q=0$ and $p^2=s=n-1=1$. Hence $M_X(p^2 s,pqH,q^2 t)=X$ in Theorem \ref{prop:Torelli}. 
\end{rmk}

Summing all of this up, we have the following:

\begin{prop}\label{lastprop}
Let $X$ be a K3 surface of degree $2d$ 
with Picard rank one and let $n>3$ be an integer.
\begin{enumerate}
\item[(1)] If $\sqrt{d(n-1)} \not \in \bbZ $, $\Mov(X^{[n]})$ is defined by $(0,0,1)$ and a primitive $v_1$, where $\vv_1$ satisfies one of the following.
\begin{enumerate}
\item If $\vv_1=(p^2 s,pqH,q^2 t)$ with $p^2 s (n-1)-q^2 t=\pm 2$, the primitivity of $\vv_1$ implies $\gcd(ps,q)=\gcd(p,t)=1$ and $P(pq \sqrt{d},p^2 s(n-1) \mp 1)$ is the generator of $S_{d,n}/\pm 1$. In this case, $\vv_1$ defines a 
Li--Gieseker--Uhlenbeck contraction, therefore $B=N$.
\item If $\vv_1=(r,cH,r(n-1))$ with $c^2 d-r^2 (n-1)=-1$, then the generator of $S_{d,n}/\pm 1$ is $P(r,c\sqrt{d})$. In this case, $\vv_1$ defines a Brill--Noether contraction, therefore $B=N$.
\item If $\vv_1=(p^2 s,pqH,q^2 t)$ with $p^2 s (n-1)-q^2 t=\pm 1$, then $P(p \sqrt{s},q \sqrt{t})$ is the generator of $S_{d,n}/\pm 1$. In this case, $\vv_1$ defines a Hilbert--Chow contraction, therefore $B=N$ if $\{s,t\}=\{1,n\}$ and $B=N/2$ otherwise.
\end{enumerate}
\item[(2)] Assume that $d(n-1)$ is a perfect square. Then $\Mov(X^{[n]})$ is defined by $(0,0,1)$ and a primitive $\vv_1$. In this case, $\vv_1$ defines a Lagrangian fibration and $B=N$.
\end{enumerate}
\end{prop}
\begin{rmk}
%
If we set $\vv=(1,0,1-n)$ then case (a) of Proposition \ref{lastprop} is governed by the conditions $(\vv_1,\vv)=\pm2$ and $\vv_1^2=0$, which amounts to finding solutions of the following Pell's equation:
\[x^2-d(n-1)y^2=1\] 
such that $(n-1)\mid x+1$. Similarly, for case (c) we have $(\vv_1,\vv)=\pm1$ and $\vv_1^2=0$, which gives rise to:
\[x^2-4d(n-1)y^2=1\] 
such that $2(n-1)\mid x+1$.
\end{rmk}

\begin{rmk}
In Proposition \ref{lastprop}, the assumption on $n$ is needed as, if $n=3$, then there may exist a Hilbert--Chow contraction with $\langle \vv_1,(1,0,-2) \rangle=\pm 2$.
\end{rmk}

\bibliographystyle{alpha}
\bibliography{ref}
\end{document}